\documentclass[12pt]{amsart}
\usepackage{amsmath,amssymb,epsfig,color}
\newtheorem{theorem}{Theorem}[section]

\newtheorem{corollary}[theorem]{Corollary}
\newtheorem{definition}[theorem]{Definition}
\newtheorem{lemma}[theorem]{Lemma}

\theoremstyle{remark}
\newtheorem{remark}[theorem]{Remark}
\newtheorem{example}[theorem]{Example}

\newcommand{\vanish}[1]{}\parskip=12pt

\newcommand{\rank}{\mbox{rank}}
\newcommand{\block}{\mbox{block}}
\newcommand{\sing}{\mbox{sing}}

\begin{document}
\title[The toric h-vector of a cubical complex]{The toric h-vector of a cubical complex in terms of noncrossing partition statistics} 
\author{Sarah Birdsong, G\'abor Hetyei}
\address{Department of Mathematics and Statistics, UNC Charlotte, 
	Charlotte, NC 28223}
\email{ghetyei@uncc.edu, sjbirdso@uncc.edu}
\date{\today}
\begin{abstract}
This paper introduces a new and simple statistic on noncrossing partitions that
expresses each coordinate of the toric $h$-vector of a cubical complex,
written in the basis of the Adin $h$-vector entries, as the total weight of all
noncrossing partitions. The same model may also be used to obtain a very
simple combinatorial interpretation of the contribution of a cubical shelling
component to the toric $h$-vector.  In this model, a strengthening
of the symmetry expressed by the Dehn-Sommerville equations
may be derived from the self-duality of the noncrossing partition lattice,
exhibited by the involution of Simion and Ullman.
\end{abstract}
\maketitle


\section*{Introduction}

There are three important $h$-vectors associated to a cubical complex,
each preserving some properties of the simplicial $h$-vector:
the triangulation $h$-vector, the toric $h$-vector, and Adin's
enigmatic cubical $h$-vector~\cite{Adin}. The triangulation
$h$-vector is used to express the Hilbert series of
the face ring of a cubical complex~\cite{Hetyei-strc}, and the toric
$h$-vector arises as a specialization of Stanley's general definition
made for all lower Eulerian posets~\cite{Stanley-hvector}. 
Adin's (long) cubical $h$-vector is obtained by a simple but mysterious algebraic operation
from his short $h$-vector, which is just the sum of the (simplicial)
$h$-vectors of the vertex figures. Adin's cubical $h$-vector has the remarkable property of
being ``smaller'' than the other two $h$-vectors in the sense that the
triangulation and toric $h$ entries are positive linear combinations of
the Adin $h$ entries. As it was observed in~\cite{Hetyei-cube}, the same
holds for any ``reasonably defined'' cubical $h$-vector that one could
invent: given any linear combination of the face numbers that is
nonnegative on the cube and weakly increases after adding any cubical
shelling component, the said invariant may be expressed as a nonnegative
linear combination of the Adin $h$ entries. 

Knowing that expressing the other two $h$-vectors in terms of the Adin
$h$ entries involves positive and, after proper scaling, integer
coefficients presents the challenge of finding a combinatorial
interpretation of these coefficients. For the triangulation $h$-vector,
such an interpretation was found by Haglund~\cite{Haglund}. The present
paper provides a combinatorial model interpreting the coefficients used
to express the toric $h$-vector. This model is a weighted enumeration  
model for noncrossing partitions, where a weight of $x$ is assigned to
all nonsingleton blocks as well as to singleton blocks and pairs of consecutive
elements in the same block, provided they are found at prescribed
positions. The model may also be used to provide a new interpretation of
the contribution of each cubical shelling component to the toric
$h$-vector. The first such combinatorial model was constructed by
Chan~\cite{Chan}; a similar second model and additional explicit
formulas may be found in~\cite{Hetyei-2nd}. The present model is
simpler and allows for a combinatorial proof of some
identities satisfied by the toric contributions of cubical shelling
components, which strengthen the symmetry expressed by the Dehn-Sommerville
equations. These equalities
were hitherto unobserved in the literature and do not seem to be obvious
when glancing at the models proposed in \cite{Chan}
and \cite{Hetyei-2nd}. In the model presented in this paper, the equalities are shown
combinatorially, using an involution introduced by Simion and
Ullman~\cite{Simion-Ullman}, which exhibits the self-duality of the
noncrossing partition lattice.  

This paper is structured as follows. In the Preliminaries, the reader is reminded of the
information that will be needed about the toric and Adin $h$-vectors. In
Section~\ref{sec:Qpoly}, the contributions $Q_{d,k}(x)$ of the Adin $h$
entries to the toric $f$ polynomial of a $d$-dimensional cubical
complex are computed. 
In Section~\ref{sec:shell}, the contribution $C_{d,i,j}(x)$ of a shelling component 
of type $(i,j)$ to the toric $f$ polynomial of a $d$-dimensional shellable cubical complex
is expressed in terms of the polynomials $Q_{d,k}(x)$.
This provides a new explicit expression for these polynomials, which were
first computed in~\cite{Hetyei-2nd}. The new combinatorial model and the
main results may be found in Section~\ref{sec:cint}. Here the
polynomials $Q_{d,k}(x)$ and $C_{d,i,j}(x)$ are expressed as the total weight of
noncrossing partitions. Finally, in Section~\ref{sec:dual}, the
Simion-Ullman involution is used to derive some immediate and some not
so trivial consequences of the Dehn-Sommerville equations for the
polynomials $Q_{d,k}(x)$ and $C_{d,i,j}(x)$. 

This work underscores the importance of the study of noncrossing
partitions in understanding the toric $h$-vector of lower Eulerian
posets. A recent result of the second author~\cite{Hetyei-short},
finding that the toric $h$-vector of a simple polytope may be expressed
using close relatives of the Narayana numbers, seems to be pointing in the same
direction.

\section{Preliminaries}
\subsection{Toric polynomials of an Eulerian poset}
\label{sec:toricH}

A poset is graded if it has a unique minimum element $\hat{0}$, 
a unique maximum element $\hat{1}$, and a rank function.  
The poset is also Eulerian if for any open interval $(x,y)$ the number 
of elements at odd and even ranks are equal.   Suppose $\hat{P}$ is 
an arbitrary Eulerian poset, and 
define $P=\hat{P} \backslash \{\hat{1}\}$.  Stanley~\cite{Stanley-hvector} 
defined the toric $h$-vector for Eulerian posets using two recursively 
defined polynomials $f(P,x)$ and $g(P,x)$ as follows:
\begin{enumerate}
\item $f(\emptyset, x)=g(\emptyset, x)=1$,
\item if $\hat{P}$ has rank $d+1 \ge1$ and $f(P,x)=k_0+k_1 x+ \cdots + k_d x^d$, 
then $g(P,x)=\sum_{i=0}^{m} (k_i - k_{i-1})x^i$ where $m=\lfloor \frac {\rank(P)}{2} \rfloor$ and $k_{-1}=0$, and
\item if $\hat{P}$ has rank $d+1 \ge1$, then 
$f(P,x)=\sum_{t \in P} g([\hat{0},t),x)(x-1)^{d-\rank(t)}$.
\end{enumerate}
Set $h_i=k_{d-i}$ for each $i$.  Then the toric $h$ polynomial is $h(P,x)=\sum h_i x^i$.  
By the generalized Dehn-Sommerville equations~\cite[Theorem 2.4]{Stanley-hvector},
$h_i = h_{d-k}$.  As a result, $h(P,x)=f(P,x)$ for Eulerian posets. 

Stanley extended the definition of the $f$ polynomial in $(3)$ to 
lower Eulerian posets.  This was possible since the definition uses
half open intervals $[\hat{0},t)$ where $t$ is any element in the poset, and
$d$ is the length of the longest chain in the poset.
 
The face poset of a {\em polyhedral complex} is a specific type of lower Eulerian poset.  
This paper uses the definition adapted to polyhedral complexes as stated by Billera, Chan, and Liu~\cite{Billera-Chan-Liu}.  

\begin{definition}[Billera-Chan-Liu]
\label{def-BCL}
Let ${\mathcal P}$ be a d-dimensional polyhedral complex and F any face of ${\mathcal P}$.  Then the toric f and g polynomials are defined by the following three rules:
\begin{enumerate}
\item $f(\emptyset, x)=g(\emptyset, x)=1$,
\item $f({\mathcal P},x)=\sum_{F \in {\mathcal P}} g(\partial F,x)(x-1)^{d-dim(F)}$, and
\item $g({\mathcal P},x)=\sum_{i=0}^{m} (k_i - k_{i-1})x^i$ where $k_i$ is the coefficient of $f({\mathcal P},x)$,
$k_{-1}=0$, and $m=\lfloor \frac {d+1}{2} \rfloor$.
\end{enumerate}
\end{definition}

Set $h(\mathcal{P},x)=\sum_i h_i x^i = x^{d+1}f(\mathcal{P},1/x)$.  
The toric $h$ polynomial is a degree $d+1$ polynomial, and the toric 
$h$-vector is comprised of the coefficients of the toric $h$ polynomial.  

If $P$ is the face complex of a convex polytope and $\partial P$ is its boundary 
complex, then $h(P,x)=g(\partial P,x)$~\cite[Corollary 1.1]{Hetyei-2nd}.
For example, the $h$ polynomial of a $d$-dimensional cube is the $g$ polynomial of its 
boundary as noted in~\cite{Billera-Chan-Liu} and proved by Stanley~\cite{Stanley-hvector};
see Table~\ref{table:gpoly}.  
Gessel~\cite{Stanley-hvector} showed that 
$g(L_d,x)=\sum_{k=0}^{\lfloor d/2 \rfloor} \frac{1}{d-k+1}{d \choose k}{2d-2k \choose d}(x-1)^k$, 
and Hetyei~\cite{Hetyei-2nd} pointed out that this was equivalent to 
$\sum_{k=0}^{\lfloor d/2 \rfloor} C_{d-k}{d-k \choose k}(x-1)^k$.

\begin{table}[h]
\caption{The $g$ polynomials of the $d$-cube for small $d$} 
\centering 
\begin{tabular}{c l} 
\hline\hline 
$d$ & $g(L_d,x)$ \\ [0.5ex] 
\hline 
$-1$ &  $1$  \\
$0$ &  $1$   \\ 
$1$ &  $1$  \\
$2$ &  $1+x$  \\
$3$ &  $1+4x$  \\
$4$ &  $1+11x+2x^2$  \\ [1ex] 
\hline 
\end{tabular}
\label{table:gpoly} 
\end{table}

\subsection{The Adin $h$-vector}
\label{sec:adinH}

Adin~\cite{Adin} called his $h$-vector for cubical complexes
the (long) cubical $h$-vector and defined it in terms of a short cubical $h$-vector, 
denoted $h_i^{(c)}$ and $h_i^{(sc)}$, respectively.  In this paper, Adin's (long) 
cubical $h$-vector will be referred to as the cubical $h$-vector or simply as the Adin $h$-vector.
Let $\mathcal{P}$ be a $d$-dimensional cubical complex with $f$-vector 
$(f_0, \ldots, f_d)$.  Then Adin's (long) $h$-vector is defined by the equations
\begin{equation}
h_i^{(sc)}=\sum_{j=0}^{i} {d-j \choose d-i}(-1)^{i-j}2^j f_j  \mbox{  for $0 \le i \le d$ and}
\end{equation}
\begin{equation}
h_i^{(sc)}=h_i^{(c)}+h_{i+1}^{(c)}  \mbox{  for $0 \le i \le d$}
\end{equation}
with the initial and last values 
\begin{equation*}
h_0^{(c)}=2^d,
\end{equation*}
\begin{equation*}
h_1^{(c)}=f_0-2^d, \mbox{ and}
\end{equation*}
\begin{equation*}
h_{d+1}^{(c)}=(-2)^d \tilde{\chi}(\mathcal{P})
\end{equation*}
where $\tilde{\chi}(\mathcal{P})$ is the Euler characteristic of 
$\mathcal{P}$; i.e., $\tilde{\chi}(\mathcal{P})=\sum_{j=0}^{d+1} (-1)^{j-1} f_{j-1}$.  
Writing the $f$-vector in terms of the cubical $h$-vector yields
\begin{equation*}
f_{j-1}=2^{1-j}\sum_{i=1}^j {d+1-i \choose d+1-j}[h_i^{(c)}+h_{i-1}^{(c)}]  \mbox{  for $1 \le j \le d+1$}.
\end{equation*}
For example, the boundary complex of the $n$-dimensional cube has a cubical 
$h$-vector of $h_0^{(c)}= \cdots = h_n^{(c)}=2^{n-1}$.

For ease of computation, normalize the Adin $h$-vector by dividing each
$h_i^{(c)}$ by $2^d$.  Then drop the superscript $(c)$.  This results 
in $h_0=1$ and the relation
\begin{equation}
\label{adinF}
f_j=2^{d-j}\sum_{i=0}^j {d-i \choose d-j}[h_{i+1}+h_i]  \mbox{  for $0 \le j \le d$}.
\end{equation}

Hetyei~\cite[Theorem 1]{Hetyei-cube} showed that any invariant $I$ of a 
$d$-dimensional cubical complex, which can be expressed as a linear 
combination of the face numbers $f_0, \ldots, f_d$, may be rewritten as a 
nonnegative linear combination of the normalized cubical $h$-vector
coordinates if and only if $I$ is nonnegative when applied to the $d$-cube and 
adding a facet of type $(1,j)$ or $(0,d)$ in any shelling does not decrease $I$.  
For the definition of shellings and cubical shelling component types, see 
Section~\ref{sec:shell}.

Due to this result, coordinates of the other $h$-vectors such as the toric $h$-vector~\cite{Hetyei-cube} and the triangulation $h$-vector~\cite{Haglund} may be
written as nonnegative linear combinations of the coordinates of the cubical
$h$-vector.  
Hence, the cubical $h$-vector is the ``smallest'' $h$-vector.  
In Section~\ref{sec:Qpoly}, this information is used to express the toric $h$ polynomial
in terms of the Adin $h$-vector.  Then several properties of the resulting 
coefficients are examined.

\subsection{Noncrossing partitions}
\label{sec:NC}

Suppose $\pi$ is a partition of $[1,d]:=\{1,\ldots,d\}$.  If whenever $1\le a<b<c<e\le d$
where $a$ and $c$  are in the same block of the partition and $b$ and $e$ are also in one
block implies that all four elements must be in the same block of $\pi$, the 
partition $\pi$ is called noncrossing.  Let $NC(d)$ denote the set of all 
noncrossing partitions of $[1,d]$.  
Then $|NC(d)|=C_d=\frac{1}{d+1}{2d \choose d}$; see~\cite{Stanley-web} or~\cite{Stanley-EC2}.

There are several ways to represent (noncrossing) partitions visually on
$[1,d]$.   
The two representations used in this paper are linear and circular.  To create a linear 
representation or an arc diagram (see Fig.~\ref{fig:arc04}), place $d$ consecutive points 
in a horizontal row.  Label the points $1$ through $d$, and connect consecutive 
elements in a block of the partition with an arc.  If there are $k$ elements in the block, 
there will be $k-1$ linked arcs representing the block in the arc diagram.  A singleton block 
will have no arcs. 

\begin{figure}[h]
\includegraphics{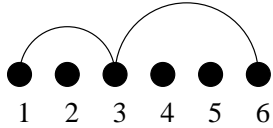}
\caption{The arc diagram of $\pi=(136)(2)(4)(5)$}
\label{fig:arc04}
\end{figure}

To create a circular representation (see Fig.~\ref{fig:circ01}), place $d$ points around a circle.  
Label the points $1$ through $d$, increasing clockwise.  Connect any two 
(cyclically) consecutive elements in a block of the partition with a chord.  In the circular representation, 
nonsingleton blocks are chords or polygons.
A partition of $[1,d]$ is noncrossing if and only if no arc or chord crosses
another in the linear and circular representations of the partition.  

\begin{figure}[h]
\includegraphics{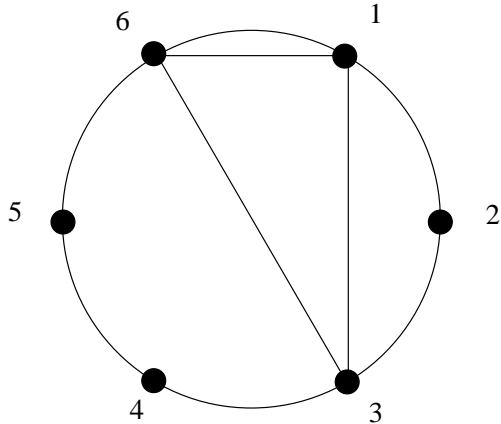}
\caption{The circular representation of $\pi=(136)(2)(4)(5)$}
\label{fig:circ01}
\end{figure}

Kreweras~\cite{Kreweras} showed that $NC(d)$ forms a lattice ordered by
refinement.  Suppose $\pi, \sigma \in NC(d)$.  Then $\pi \le \sigma$ under 
refinement if every block of $\pi$ is contained in a block of $\sigma$.  For
example, $(12)(347)(56) \le (1256)(347)$.  Kreweras also used what later 
became known as the Kreweras complement to show that the lattice formed
by NC(d) is self-dual.

Simion and Ullman~\cite{Simion-Ullman} gave an alternative proof for the self-duality of $NC(d)$,
which used the circular representation of partitions and the involution they defined, namely 
$\alpha: NC(d) \to NC(d)$.  This involution will be used throughout the rest of the paper.

The involution $\alpha(\pi)$ is defined in 
the following way.  Take any $\pi \in NC(d)$, and represent it circularly.  Label the 
midpoint of the arc determined by $(d-1,d)$ as $1'$.  Subdivide each arc $(i,i+1)$ 
by placing a new point between the existing points on the circle. 
Label the new point $(d-i)'$.  $1'$ through $d'$ increase counterclockwise, and $d'$ is the 
midpoint of the arc $(d,1)$.  Then $\alpha(\pi)$ is represented on the same circle as $\pi$ by 
the coarsest noncrossing partition of $[1',d']$ that does not intersect any of the chords of $\pi$.
In Fig.~\ref{fig:circ03}, $\pi$ is represented by the solid lines and filled circles 
and $\alpha(\pi)$ by the dashed lines and open circles. 

\begin{figure}[h]
\input{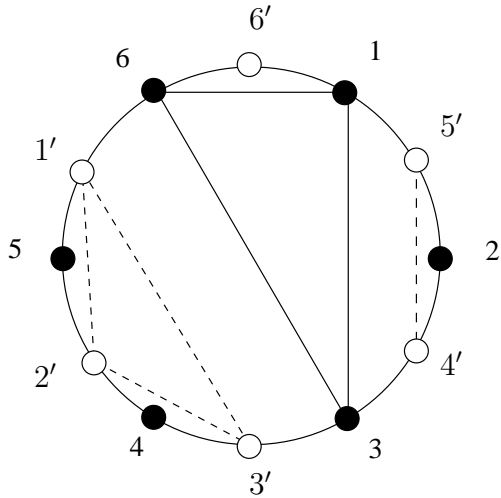}
\caption{The circular representation of $\pi=(136)(2)(4)(5)$ and $\alpha(\pi)=(123)(45)(6)$}
\label{fig:circ03}
\end{figure}

See Simion and Ullman~\cite{Simion-Ullman} for an alternative version of 
the definition of $\alpha(\pi)$, which does not use the circular representation 
of a partition.
For the rest of this paper, $i$ will be used to represent $i'$ as an element 
of $\alpha(\pi)$ unless it is unclear in the context as to whether the 
element $i$ is in $\alpha(\pi)$ or $\pi$.  

For more information on noncrossing partitions and their applications, see~\cite{Comtet}, 
\cite{Kreweras}, \cite{Nica-Speicher}, \cite{Simion}, \cite{Simion-Ullman}, \cite{Stanley-EC1},
\cite{Stanley-NC}.

\section{The toric contribution of the Adin h-vector}
\label{sec:Qpoly}

Let ${\mathcal P}$ be a $d$-dimensional cubical complex.  
In this section, the toric contribution of the normalized Adin $h$-vector will be determined
using the information from Section~\ref{sec:adinH}.  
Start with Definition~\ref{def-BCL} for the toric $f$ polynomial of $\mathcal{P}$.

\begin{equation}\label{h_vec_1}
f({\mathcal P},x) = \sum_{F \in {\mathcal P}} g(\partial F, x)(x-1)^{d-dim F}
\end{equation}
where $F$ is a face of ${\mathcal P}$.  Then $F$ is a $j$-cube for some $0 \le j \le d$ or $F=\emptyset$.  Hence,
$$ f({\mathcal P},x) = (x-1)^{d+1} + \sum_{j=0}^{d} f_j \cdot (x-1)^{d-j}g(L_j,x) $$
where $(x-1)^{d+1}$ is the contribution when $F=\emptyset$.  Applying Equation~(\ref{adinF}), 
\begin{equation*}
f({\mathcal P},x)=(x-1)^{d+1} + \sum_{j=0}^{d} 2^{d-j}\sum_{i=0}^{j} {d-i \choose d-j}[h_{i+1}+h_i]  (x-1)^{d-j}g(L_j,x). 
\end{equation*}
Since $h_0=1$, collecting the contribution of each $h_i$ yields
\begin{equation} \label{h_vec}
\begin{split}
f({\mathcal P},x)& =\left[(x-1)^{d+1} + \sum_{j=0}^{d} 2^{d-j}{d \choose j}(x-1)^{d-j}g(L_j,x) \right] +  h_{d+1}\cdot g(L_d,x) \\
& \quad\quad  +\sum_{i=1}^{d} h_i \sum_{j=i-1}^{d} 2^{d-j}\left[{d-i \choose d-j}+{d+1-i \choose d-j}\right](x-1)^{d-j}g(L_j,x). \\
\end{split}
\end{equation}
Thus, $(x-1)^{d+1}$ will be part of the toric contribution of $h_0$.

\begin{definition}
Suppose ${\mathcal P}$ is a $d$-dimensional cubical complex.  Let $Q_{d,k}(x)$ be the polynomial which gives the toric contribution of the normalized Adin $h$-vector to $f({\mathcal P},x)$, namely 
\begin{equation}\label{c_vec}
f({\mathcal P},x) = \sum_{k=0}^{d+1} h_k \cdot Q_{d,k}(x).
\end{equation}
\end{definition}

By Equation~(\ref{h_vec}), the polynomials $Q_{d,k}(x)$ have an explicit formula:
\begin{equation}
\label{form:Qpoly}
Q_{d,k}(x)=
\left\{
\begin{array}{ll}
(x-1)^{d+1} + \sum_{j=0}^{d} 2^{d-j}{d \choose j}(x-1)^{d-j}g(L_j,x) & \mbox{if $k=0$,}  \\
\sum_{j=k-1}^{d} 2^{d-j}\left[{d-k \choose d-j}+{d+1-k \choose d-j} \right](x-1)^{d-j}g(L_j,x) & \mbox{if $1 \le k \le d$,}  \\
g(L_d,x) & \mbox{if $k=d+1$.} \\
\end{array}
\right.
\end{equation}
Table~\ref{table:Qpoly} in Section~\ref{sec:cint} lists $Q_{d,k}(x)$ for
small $d$. 
For future reference, note that, by Pascal's identity,
\begin{equation}
\label{eq:pascal}
Q_{d,k}(x)=\sum_{j=k-1}^{d} 2^{d-j}\left[{d-k \choose d-j-1}+2{d-k
    \choose d-j} \right](x-1)^{d-j}g(L_j,x) \quad \mbox{for $1 \le k \le d$.}
\end{equation}

\section{Contribution of shelling components to the toric $h$-vector}
\label{sec:shell}

\subsection{Definition of shelling and shelling component types}

Let $\mathcal{P}$ be a $d$-dimensional pure cubical complex where $\{F_1,\ldots,F_m\}$ are its facets.  
A {\em shelling} of $\mathcal{P}$ is a particular way of ordering the facets  such that for every $1 \le k \le m$,
the intersection $F_k \cap (F_1 \cup \ldots \cup F_{k-1})$ is the union of $(d-1)$-faces homeomorphic to a ball or sphere.
See~\cite[Definition 8.1]{Ziegler} for a definition of a shelling of the boundary of a polytope.
The {\em cubical shelling component} $F_k$ has {\em type} $(i,j)$ if the above intersection is the union of 
$i$ antipodally unpaired $(d-1)$-faces and $j$ pairs of antipodal $(d-1)$-faces. 
This was stated by Chan~\cite{Chan} and Adin~\cite[Theorem 5(iii)]{Adin}.  
For a proof, see~\cite[Lemma 3.3]{Ehrenborg-Hetyei}.

\begin{lemma}[Ehrenborg-Hetyei]
\label{lem:shell}
The ordered pair $(i,j)$ is the type of a shelling component in a shelling of a cubical 
$d$-complex if and only if one of the following holds:
\begin{itemize}
\item[(i)] $i=0$ and $j=d$; or 
\item[(ii)] $0<i<d$ and $0\leq j\leq d-i$.
\end{itemize}
Furthermore, in case $(i)$, the shelling component is homeomorphic to a $(d-1)$-sphere;
and in case $(ii)$, the shelling component is homeomorphic to a $(d-1)$-ball.
\end{lemma}

A direct result of this lemma is that the first shelling component has type $(0,0)$, and
for cubical spheres, the last shelling component has type $(0,d)$.

Let $c_{i,j}$ be the number of type $(i,j)$ shelling components of $\mathcal{P}$.  In 
particular, $c_{0,0}=1$.  The vector $(\ldots, c_{i,j}, \ldots)$ is
called the {\em $c$-vector} of the shelling.  The $c$-vector is not unique to the 
cubical complex; instead, it depends on which shelling is chosen. 
However, different shellings may have the same $c$-vector.

\subsection{The toric contribution of $c_{i,j}$}
\label{sec-toricC}

Consider a shelling $F_1, \ldots, F_m$ of $\mathcal{P}$, and   
let $(i,j)$ be the type of the $t^{th}$ facet $F_t$.
According to Adin, $\sum_{k} h_kx^k = \sum_{t} \Delta_{t}h(x)$,
where $h_k$ is the normalized cubical $h$-vector and
$\Delta_{t}h(x)$ is the contribution from $F_t$.  
Adin~\cite[Equation (33)]{Adin} defines this contribution as
$$\Delta_{t}h(x)=
\left\{
\begin{array}{ll}
(\frac{1}{2})^i x^{j+1} (1+x)^{i-1} & \mbox{if $i \ge 1$,} \\
1  & \mbox{if $i=0$ and $j=0$} \\
x^{d+1} & \mbox{if $i=0$ and $j=d$.} \\   
\end{array}
\right.
$$
When $i\ge 1$, Lemma~\ref{lem:shell} implies that $1\le i \le d-j$ and $0\le j \le d-1$.  Thus,  
\begin{equation*}
\begin{split}
\sum_{k} h_kx^k&=1+c_{0,d}\cdot x^{d+1}+\sum_{j=0}^{d-1}\sum_{i=1}^{d-j} c_{i,j}\left(\frac{1}{2}\right)^i x^{j+1} (1+x)^{i-1}\\
&=1+c_{0,d}\cdot x^{d+1}+\sum_{j=0}^{d-1}\sum_{i=1}^{d-j} c_{i,j}\left(\frac{1}{2}\right)^i x^{j+1} \sum_{m=0}^{i-1}{i-1 \choose m}x^m.\\
\end{split}
\end{equation*}
Set $k=m+j+1$, and reorder the summations to get
\begin{equation*}
\begin{split}
\sum_{k} h_kx^k&=1+c_{0,d}\cdot x^{d+1}+\sum_{k=1}^{d}\sum_{j=0}^{k-1}\sum_{i=k-j}^{d-j} c_{i,j}\left(\frac{1}{2}\right)^i {i-1 \choose k-j-1}x^k.\\
\end{split}
\end{equation*}
This is equivalent to 
\begin{equation}
\label{form-hc}
h_k=\sum_{j=0}^{k-1} \sum_{i=k-j}^{d-j} \left(\frac{1}{2}\right)^i {i-1 \choose k-1-j}c_{i,j} \mbox{ for $1\le k \le d$.}
\end{equation}
Apply Equations (\ref{c_vec}) and (\ref{form-hc}) to get the contribution of $c_{i,j}$ in terms of 
the polynomials $Q_{d,k}(x)$.
$$\sum_{k=1}^{d} h_k \cdot Q_{d,k}(x)=\sum_{j=0}^{d-1} \sum_{i=1}^{d-j} c_{i,j} \sum_{k=j+1}^{i+j} \left(\frac{1}{2}\right)^i {i-1 \choose k-1-j}Q_{d,k}(x)$$
Define $C_{d,i,j}(x)$ to be the toric contribution of all shelling components of type $(i,j)$.  Hence,
\begin{equation}
\label{equ-c}
C_{d,i,j}(x)=\sum_{k=j+1}^{i+j} \left(\frac{1}{2}\right)^i {i-1 \choose k-1-j}Q_{d,k}(x)  \quad\mbox{for $i>0$}.
\end{equation}
In particular, for $i=1$, we have
\begin{equation}
\label{eq:CQ}
C_{d,1,j}(x)=\frac{1}{2}Q_{d,j+1}(x).
\end{equation}
Additionally, $C_{d,0,0}(x)=1$ and $C_{d,0,d}(x)=c_{0,d}\cdot
x^{d+1}$. 

\section{A combinatorial interpretation}
\label{sec:cint}
\subsection{A weight statistic for noncrossing partitions}
\label{sec:NC-wt}

In a noncrossing partition, consider three special types of elements.
The first are {\em singleton} elements.  
In a nonsingleton block, call the largest element of the block the {\em last element}.  
Obviously, the number of last elements in a partition is exactly the number of its nonsingleton blocks.  
Lastly, suppose both $k$ and $k+1$ are in the same block of a partition
(where addition is performed modulo $d$); 
in this case, call the element $k$ an {\em antisingleton} element.  This naming convention 
follows from the self-duality of the lattice of $NC(d)$.  See Lemma~\ref{Lem-anti} 
and its proof for more details.

Let $\block(\pi)$ equal the number of nonsingleton blocks in the partition $\pi$ and
$\sing(\pi)$ equal the number of singleton elements in $\pi$.
Then the total number of components of the partition is $|\pi|=\block(\pi)+\sing(\pi)$.  
For example, if $\pi=(136)(2)(4)(5)$, then $\block(\pi)=1$, $\sing(\pi)=3$, and $|\pi|=4$.

In Simion and Ullman's~\cite{Simion-Ullman} proof of the 
self-duality of the lattice of $NC(d)$, the involution $\alpha$ was defined such that 
if $\pi \in NC(d)$ has $k$ blocks (either singleton or 
nonsingleton blocks) then $\alpha(\pi)$ has $d+1-k$ blocks.  
Lemma~\ref{Lem-sum} follows immediately.

\begin{lemma}
\label{Lem-sum}
Let $\pi \in NC(d)$ for some $d>0$.  Then $|\pi|+|\alpha(\pi)|=d+1$. 
\end{lemma}

\begin{lemma}
\label{Lem-anti}
Let $\pi \in NC(d)$.  Then $k\in [1,d-1]$ is an antisingleton element of $\pi$ if and 
only if $d-k$ is a singleton element of $\alpha(\pi)$. 
\end{lemma}
\begin{proof}
Let $k$ be an antisingleton of $\pi$.  Then $k+1$ must be in the same 
block of $\pi$ as $k$, and in the circular representation of $\pi$, $k$ and 
$k+1$ are connected by a chord.  By definition, $\alpha(\pi)$ is the coarsest 
noncrossing partition on $[1,d]$ whose chords do not cross any chord of $\pi$.  
Since $d-k$ is the element of $\alpha(\pi)$ located between $k$ and 
$k+1$ of $\pi$, $d-k$ must be a singleton of $\alpha(\pi)$.

The converse is proved similarly.
\end{proof}

\begin{remark}
\label{anti-d}
Following the same method as in the proof of Lemma~\ref{Lem-anti}, 
one can prove that $d$ is a singleton element of $\alpha(\pi)$ if and only
if $d$ is an antisingleton element of $\pi$.
\end{remark}

Definition~\ref{def-wtS} below will define a weight function in terms of
a family $\mathcal{S}$ of $i$ pairwise disjoint subsets of $[1,d]$ such that each element in $\mathcal{S}$ is a set of consecutive integers of $[1,d]$
in circular order.  An element of $\mathcal{S}$ is denoted by
$[k,l]:=\{k,\ldots,l\}$ and is called an {\em interval} of  
$\mathcal{S}$.  In particular, if $k=l$, the interval consists of a
single element.  If $1 \le k < l \le d$, then $[k,l]$  
is a set of increasing consecutive integers.  Call
$[k,l]=\{k,\ldots,d,1,\ldots,l\}$ a {\em wrapped} interval when $k>l$.  
Notice the last element of a wrapped interval, namely $l$, will not be
the largest element of the interval. Let $[1,d]^*$ denote the special wrapped interval consisting of all of the elements $1$ through 
$d$; whereas, $[1,d]$ is the non-wrapped interval. 
This distinction will be needed to treat the case where $d$ is an
antisingleton; see the last sentences before Example~\ref{ex:ps} below.

List the elements of $\mathcal{S}$ in the order 
$\mathcal{S}=\{[k_1,l_1],\ldots,[k_i,l_i]\}$ where $1\le k_1\le
l_1<k_2\le l_2<\ldots<k_i\le d$ and either $k_i\le l_i\le d$ or $1\le
l_i<k_1$.   
If $1\le l_i<k_1$, then the last interval in $\mathcal{S}$ is wrapped.
Also, only the last interval may be wrapped. 
Define $j$ to be the number of elements of $[1,d]$ that are not
contained in any interval of $\mathcal{S}$. 

$\mathcal{S}$ may be represented visually using the same circular representation that was defined
in Section~\ref{sec:NC}.  The intervals of $\mathcal{S}$ correspond to sets of consecutive elements;
see Fig.~\ref{fig:circS01}.  Let $B$ be the union of the regions bounded by the arcs corresponding 
to the intervals of $\mathcal{S}$.  Call the remaining region $A$.
In other words, for each $[k,l] \in \mathcal{S}$, draw an arc such that the elements of $[k,l]$ are on one side 
of the arc and all other elements of $[1,d]$ are on the other side of the arc.

\begin{figure}[h]
\includegraphics{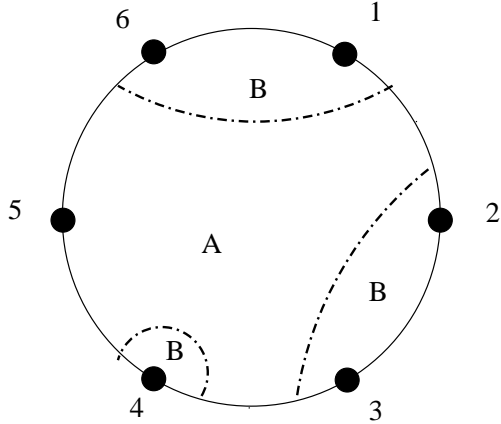}
\caption{The circular representation of $\mathcal{S}=\{[2,3],[4],[6,1]\}$ where $d=6$ and $[1,6]-\mathcal{S}=\{[5]\}$}
\label{fig:circS01}
\end{figure}

Let $[1,d]-\mathcal{S}$ be the set of intervals determined by the longest sequence of 
consecutive elements in $[1,d]$ located between the intervals of $\mathcal{S}$ in circular order.  
If $\mathcal{S}$ is defined as above, then 
$[1,d]-\mathcal{S}=\{[l_1+1,k_2-1],[l_2+1,k_3-1],\ldots,[l_i+1,k_1-1]\}$,
where $[l_n+1,k_{n+1}-1]$ exists for $1\le n <i$ if and only if $l_n<k_{n+1}-1$, and $[l_i+1,k_1-1]$
exists if and only if $l_i+1 \not \equiv k_1$ in circular order.  Also, if $l_i=d$, then $l_i+1\equiv1$, 
and if $k_1=1$, then $k_1-1\equiv d$.  If $d\le l_i<k_1-1$, then $[l_i+1,k_1-1]$ is the first interval of 
$[1,d]-\mathcal{S}$ in the same sense as in the ordering of the intervals of $\mathcal{S}$; 
otherwise, it is considered the last interval in the set $[1,d]-\mathcal{S}$.

Notice that $[1,d]-\mathcal{S}$ has the same number of intervals as $\mathcal{S}$ exactly when,
in circular order, each interval of $\mathcal{S}$ is directly followed by an element of $[1,d]$ that is not in 
any interval of $\mathcal{S}$; i.e., $l_n+1<k_{n+1}$ for $n\ge1$.  Otherwise, 
$[1,d]-\mathcal{S}$ will contain fewer intervals than $\mathcal{S}$.  See Fig.~\ref{fig:circS01}.
The definition of $j$ implies that $[1,d]-\mathcal{S}$ will always contain a total of 
$j$ elements from $[1,d]$ in its intervals.

For each $\mathcal{S}=\{[k_1,l_1],\ldots,[k_i,l_i]\}$, consisting of $i$
intervals and containing a total of $d-j$ elements, we will also
consider the related family of intervals
$\mathcal{S}'=\{[d-k_i+1,d-l_{i-1}],\ldots,[d-k_2+1,d-l_1],[d-k_1+1,d-l_i]\}$. 
Since $k_1\le l_1<k_2\le l_2<\ldots<k_i$, then $d-k_i+1\le
d-l_{i-1}<d-k_{i-1}+1\le d-l_{i-2}<\ldots<d-k_1+1$. The family
$\mathcal{S}'$ contains $i$ pairwise disjoint intervals and a total of
$i+j$ elements.  See below for a justification. 
In particular, if $\mathcal{S}=\emptyset$, then
$\mathcal{S}'=\{[1,d]^*\}$ and vice versa. 

$\mathcal{S}$ and $\mathcal{S}'$ can be represented visually on the same
circle.  Place $[1,d]$ and $[1',d']$ on a circle as in
Section~\ref{sec:NC}.  For each $[k_n,l_n]\in {\mathcal S}$, insert an
arc, whose endpoints are adjacent to $k_n$ and $l_n$, in the way
described above. Make sure the endpoints of the arc are near enough
to $k_n$ and 
$l_n$ such that the adjacency relation is not destroyed when the
elements of $[1',d']$ are marked.  Call the union of the regions determined by
these arcs $B$.  Let the remaining region be $A$. $\mathcal{S}'$ is the
set of intervals determined by the longest list of consecutive elements
of $[1',d']$ in region $A$,  
which do not skip over any region of $B$.  See Fig.~\ref{fig:circS02}.
Notice that in this visual representation, ${\mathcal S}'$ is
represented on the disjoint copy $[1',d']$ of $[1,d]$. 

\begin{figure}[h]
\input{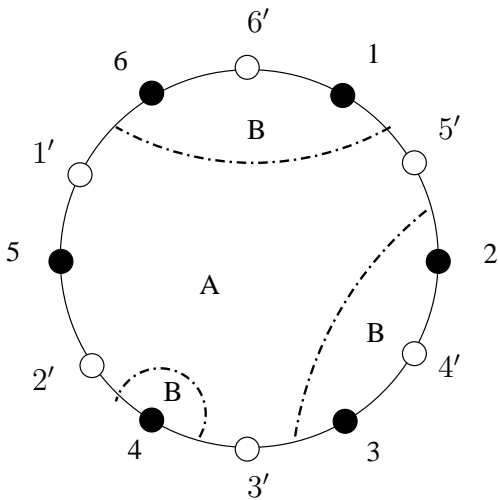}
\caption{The circular representation of $\mathcal{S}=\{[2,3],[4],[6,1]\}$ with $\mathcal{S}'=\{[1,2],[3],[5]\}$} 
\label{fig:circS02}
\end{figure}

Using this visual representation the family
${\mathcal S}$ of intervals on $[1,d]$ may be associated to the set
$\phi({\mathcal S})$ of all elements of $\{1,1',\ldots,d,d'\}$ that
belong to the union of regions $B$. The map 
${\mathcal S}\mapsto \phi({\mathcal S})$ is a bijection between the
families of pairwise disjoint intervals on $[1,d]$ and
those subsets of $\{1,1',\ldots,d,d'\}$  which consist of strings
of consecutive elements where the first and last element of each string
must be from $\{1,2,\ldots,d\}$. In particular, 
$\phi(\{[1,d]^*\})=\{1,1',\ldots, d,d'\}$ and
$\phi(\{[1,d]\})=\{1,1',\ldots, d,d'\}\setminus \{d'\}$. Define the
map $\phi'$ on the families of pairwise disjoint intervals on $[1',d']$
in a completely analogous way: these will be in bijection with all
subsets of $\{1,1',\ldots,d,d'\}$  which consist of
strings of consecutive elements where the first and last element of each
string must be from $\{1',2',\ldots,d'\}$. If ${\mathcal
  S}'$ is regarded as a family of intervals on $[1',d']$, then
$\phi'({\mathcal S}')$ is the complement $\{1,1',\ldots,d,d'\}\setminus
\phi({\mathcal S})$ of $\phi({\mathcal S})$. 

 Define $\beta$ to be the operation that takes a given $\mathcal{S}$ and
 transforms it into the associated $\mathcal{S}'$; i.e.,
 $\beta(\mathcal{S})=\mathcal{S}'$ (now consider them both as
 families of intervals on $[1,d]$).  Then $\beta$ is an involution since
 given $\mathcal{S}$ and $\mathcal{S}'$ as defined above
 $\beta(\mathcal{S}')=
 \{[d-(d-k_1+1)+1,d-(d-l_1)],\ldots,[d-(d-k_{i-1}+1)+1,d-(d-l_{i-1})],[d-(d-k_i+1)+1,d-(d-l_i)]\}=\mathcal{S}$. Using
 the visual representation described above, the same fact also follows
 from the observation that taking complements is an involution on the
 set of subsets of $\{1,1',\ldots,d,d'\}$.

Obviously, $\mathcal{S}'$ consists of $i$ intervals.  To see that $\mathcal{S}'$ contains a total of $i+j$
elements, consider the visual representation described above. The set
$\phi({\mathcal S})$ has $i$ maximal strings of consecutive elements,
each string contains one more element of $[1,d]$ than of
$[1',d']$. Since $\phi({\mathcal S})$ contains $d-j$ elements of
$[1,d]$, it contains $d-j-i$ elements of $[1',d']$. Therefore the complement
$\phi'({\mathcal S}')$ must contain $d-(d-i-j)=i+j$ elements of
$[1',d']$. 

\begin{lemma}
Let $\mathcal{S}$ be as defined above and $\mathcal{S} \neq \{[1,d]^*\}$.  
If $l_i\neq d$, then exactly one interval of $\mathcal{S}$ and $\mathcal{S}'$ is wrapped.
If $l_i=d$, then no interval of $\mathcal{S}$ or $\mathcal{S}'$ is wrapped.
\end{lemma}
\begin{proof}
By the numbering convention of the intervals of $\mathcal{S}$, only the last interval $[k_i,l_i]$ may be wrapped.
This is also true for $\mathcal{S}'$.
Let $l_i \neq d$.  Suppose $[k_i,l_i]$ is not wrapped. Then $1\le k_1<l_i<d$, which implies that
$1\le d-l_i<d-k_1+1\le d$.  This is equivalent to the last interval of $\mathcal{S}'$ $[d-k_1+1,d-l_i]$ being wrapped.
These implications are all reversible; hence, $[k_i,l_i]$ is not wrapped if and only if $[d-k_1+1,d-l_i]$ is wrapped.
The result that $[k_i,l_i]$ is wrapped if and only if $[d-k_1+1,d-l_i]$ is not wrapped is proved similarly.

If $l_i=d$, then the last interval of $\mathcal{S}$ is $[k_i,d]$, and  
the last interval of $\mathcal{S}'$ is $[d-k_1+1,d]$.  Neither of which is wrapped.
\end{proof}

If $\mathcal{S}=\{[1,d]^*\}$, then $\mathcal{S}'=\emptyset$.  Obviously, in this situation, 
exactly one of the intervals of $\mathcal{S}$ and $\mathcal{S}'$ is wrapped.

Given some partition of $[1,d]$ and an $\mathcal{S}$ as defined above,
we say that a singleton element or a last element is 
{\em in $\mathcal{S}$} if the element is contained in some interval of
$\mathcal{S}$. We also say that an antisingleton element $k$ is {\em in $\mathcal{S}$} when the pair $\{k,k+1\}$ is contained in a single interval of 
$\mathcal{S}$.  In the special case when $d$ is an antisingleton element
of the partition, we say it is {\em in $\mathcal{S}$} if and only if
the last interval of $\mathcal{S}$ is wrapped. Notice $d$ can be an
antisingleton in $[1,d]^*$ but not in $[1,d]$. 

\begin{example}
\label{ex:ps}
Let $\pi=(136)(2)(4)(5)$ and $\mathcal{S}=\{[2,3],[4],[6,1]\}$.  Then 
$\alpha(\pi)=(123)(45)(6)$ and $\mathcal{S}'=\{[1,2],[3],[5]\}$.  See Fig.~\ref{fig:circ04}.

\begin{figure}[h]
\input{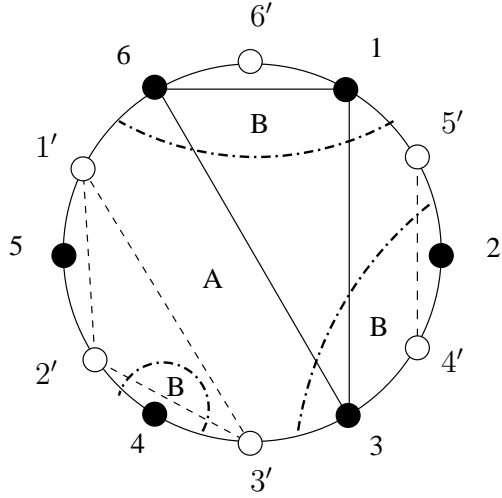}
\caption{The circular representation of $\pi=(136)(2)(4)(5)$ with $\mathcal{S}=\{[2,3],[4],[6,1]\}$ 
and $\alpha(\pi)=(123)(45)(6)$ with $\mathcal{S}'=\{[1,2],[3],[5]\}$}
\label{fig:circ04}
\end{figure}
\end{example}

\begin{definition}
\label{def-wtS}
Let $\pi \in NC(d)$ and $\mathcal{S}$ as defined above.  Define the weight function 
$wt_{\mathcal{S}}(\pi)$ as follows.  Singleton and antisingleton elements have a weight of $x$ if they 
are in $\mathcal{S}$ and a weight of $1$ otherwise.  Nonsingleton blocks have a weight of $x$.  
\end{definition}

If $\mathcal{S}$ consists of a single interval, namely
$\mathcal{S}=\{[k,d]\}$ for some $k$, use the following abbreviated 
notation of the weight function.

\begin{definition}
\label{def-wt}
Let $\pi \in NC(d)$.  If $1 \le k \le d$, define $wt_k(\pi)=wt_{\{[k,d]\}}(\pi)$.
Define $wt_0(\pi)=wt_{\{[1,d]^*\}}(\pi)$, and define $wt_{d+1}(\pi)=wt_\emptyset(\pi)$.
\end{definition}

A direct consequence of this definition is that $wt_{d+1}(\pi)=x^{\block(\pi)}$
and $wt_0(\pi)=x^{|\pi|+\sing(\alpha(\pi))}$.  By Lemma~\ref{Lem-sum}, 
the second weight function may be rewritten as $w_0(\pi)=x^{d+1-\block(\alpha(\pi))}$.  

Using~\cite[Lemma 5.4]{Hetyei-2nd}, given below, $\sum_{\pi \in NC(d)} w_{d+1}(\pi) = g(L_d, x)$.

\begin{lemma}[Hetyei]
\label{Thrm-wt}
$[x^k]g(L_d,x)$ is the number of noncrossing partitions on $[1,d]$ with 
exactly $k$ nonsingleton blocks.
\end{lemma}

Let $\sing_{\mathcal{S}}(\pi)$ denote the number of singleton elements of $\pi$ in $\mathcal{S}$.  
By Lemma~\ref{Lem-anti}, $\sing_{[1,d]-\mathcal{S}'}(\alpha(\pi))$ is the number
of antisingletons of $\pi$ in $\mathcal{S}$.  Thus, Definition~\ref{def-wtS} yields the 
explicit formula
\begin{equation}
\label{form-wtS}
wt_{\mathcal{S}}(\pi)=x^{\block(\pi)+\sing_{\mathcal{S}}(\pi)+\sing_{[1,d]-\mathcal{S}'}(\alpha(\pi))},
\end{equation}
and for $1 \le k \le d$, Definition~\ref{def-wt} gives the formula 
\begin{equation}
\label{form-wt}
wt_k(\pi)=x^{\block(\pi)+\sing_{\{[k,d]\}}(\pi)+\sing_{\{[1,d-k]\}}(\alpha(\pi))}. 
\end{equation}

\subsection{$C_{d,i,j}(x)$ as the total weight of objects}
\label{sec-Cweight}

\begin{lemma}
\label{lem-halfQ}
Let $\pi \in NC(d)$, and suppose $\mathcal{S}=\{[n,m]\}$ where $[n,m]$ is wrapped if $n>m$.
Let $d-j$ be the number of elements in $[n,m]$.  Hence, $0 \le j \le d-1$.  Then
\begin{equation}
C_{d,1,j}(x)=\sum_{\pi \in NC(d)} wt_{\mathcal{S}}(\pi)
\end{equation}
\end{lemma}
\begin{proof}
Definition~\ref{def-wtS} may be rephrased as follows.  Consider the special objects: 
singleton and antisingleton elements in $\mathcal{S}$ as well as all nonsingleton blocks in $\pi$.  
Assign the weight $x$ to each special object in $\pi$.  All others receive a weight of $1$.

Next, consider a related weight function $wt_{\mathcal{S}}'$, where any special object is assigned a 
weight of $x+1$.  All others have weight $1$.  This modified weight function 
$wt_{\mathcal{S}}'$ differs from $wt_{\mathcal{S}}$ in that there exists the option to  
choose whether or not to mark each special object.  
Each marked object has weight $x$, and each unmarked object has weight $1$. 
Then $wt_{\mathcal{S}}'(\pi)$ is computed by replacing every $x$ in
$wt_{\mathcal{S}}(\pi)$ with $x+1$.  It suffices to show the 
equivalent equality $C_{d,1,j}(x+1)=\sum_{\pi \in NC(d)}
wt_{\mathcal{S}}'(\pi)$. 

To compute $\sum_{\pi \in NC(d)} wt_{\mathcal{S}}'(\pi)$, choose $\pi$ in $NC(d)$, and mark as many
nonsingleton blocks, antisingleton elements in $\mathcal{S}$, and last elements where all other elements
in the block are marked and contained in one interval of $\mathcal{S}$ as desired.

Define {\em type a} elements to be the marked antisingletons of $\pi$.  Define {\em type b}
elements to be marked last elements of a block of $\pi$ whose other elements are all type $a$
and the entire block is completely contained in one interval of $\mathcal{S}$.  Marked singletons
are type $b$ elements.

Remove all type $a$ and type $b$ elements from $\pi$.  Let $k$ be the number of elements remaining in the partition.
Call the noncrossing partition formed by these $k$ elements $\pi_k$.  Any marked object left in 
$\pi_k$ must be a nonsingleton block.   By Lemma~\ref{Thrm-wt}, the weight of all possible $\pi_k$ in $NC(k)$ 
is $g(L_k,x+1)$.  Since type $a$ and type $b$ elements are only located in 
$\mathcal{S}$, there are at most $d-j$ of them.  Since $d-k$ is the number of removed 
type $a$ and type $b$ elements, $j \le k \le d$.

Next, reinsert $d-k$ type $a$ and type $b$ elements into $\pi_k$, and count the number of ways
this can be done.  If only the position and order of the marked type $a$ and type $b$ elements is known, 
the original $\pi$ can be recovered from $\pi_k$.   
Consider the linear representation of $\pi_k$.  Insert elements at each position 
where type $a$ and type $b$ elements are known to be located.  For the type $b$ 
elements, do nothing else.  The inserted antisingleton element joins the block to which the element directly 
following it belongs.  Thus, excluding the situation where a type $b$ element is directly preceded by a type 
$a$ element, all type $b$ elements are singletons.  The original partition $\pi$ is constructed from this newly 
created arc diagram.

For fixed $k$, one can determine where to insert the $d-k$ type $a$ and type $b$ elements.
Type $b$ elements may be inserted anywhere in $\mathcal{S}$.  
Type $a$, or antisingleton, elements may be inserted anywhere in $\mathcal{S}$ except at position $m$.  

Unless an inserted antisingleton element is placed immediately prior to a singleton element in $\pi_k$, 
the elements of $\pi_k$ will have the same role in $\pi$ as they did in $\pi_k$.  
If an antisingleton element is inserted before a singleton element in $\pi_k$, the singleton element
will change to a last element of a nonsingleton block.  Singleton elements cannot be marked in $\pi_k$, so the new 
nonsingleton block is still unmarked.  Thus, this change will not affect the overall weight. 

One may also insert an arbitrary number of consecutive type $a$ and type $b$ elements.   
Each inserted element may be either type $a$ or type $b$ so long as it is not located at position $m$.
Recall, only type $b$ elements may be placed at the last position of any interval in $\mathcal{S}$.

If $m$ is a type $b$ element, then the remaining $d-k-1$ type $a$ and type $b$ 
elements are placed in the other $d-j-1$ positions of $[n,m]$, and the total number of 
ways to arrange these elements is ${d-j-1 \choose d-k-1} 2^{d-k-1}$.  
If $m$ is not a type $b$ element, the last element of $[n,m]$ is an unmarked element.  
The element at position $m$ could be in a marked nonsingleton block, but it will never be a type $a$ element.
In this situation, all type $a$ and type $b$ elements are in the remaining $d-j-1$ positions of $\mathcal{S}$.  
The total number of ways to arrange these elements is ${d-j-1 \choose d-k} 2^{d-k}$.

Summing over all partitions of $NC(d)$ yields a total weight of
$$\sum_{k=j}^{d} \left[ {d-j-1 \choose d-k-1}2^{d-k-1} + {d-j-1 \choose d-k}2^{d-k} \right] x^{d-k} g(L_k,x+1),$$
where $x^{d-k}$ is the weight of the type $a$ and type $b$ elements, and $g(L_k,x+1)$ 
gives the weight of all $\pi_k$.  Thus,
$$\sum_{\pi \in NC(d)} wt_{\mathcal{S}}'(\pi)=\sum_{k=j}^{d} \left[ {d-j-1 \choose d-k-1} + {d-j-1 \choose d-k}2 \right]2^{d-k-1} x^{d-k} g(L_k,x+1).$$
By Equation (\ref{eq:pascal}) the last sum equals
$\frac{1}{2}Q_{d,j+1}(x+1)$, which, by Equation (\ref{eq:CQ}), is the
same as $C_{d,1,j}(x+1)$. 
\end{proof}

Note $k=d$ implies that $d-k=0$, yielding the special case where there are no type $a$ or type $b$ elements.  
At the other extreme, $k=j$ means that $d-k=d-j$, or every position in $\mathcal{S}$ is a type $a$ or type $b$ element.

The following identity will be used in Theorem~\ref{thrm-cwt}, the main result of this section.

\begin{lemma}
\label{lem-little}
Let $d>0$, $i>0$, $0 \le j \le d-i$, and $j \le k \le d$.  Then
$$\sum_{\ell=0}^{i}{i \choose \ell}{d-i-j \choose d-k-\ell}\cdot 2^{i-\ell}=\sum_{m=0}^{i-1}{i-1 \choose m}\left[2{d-m-j-1 \choose d-k}+{d-m-j-1 \choose d-k-1}\right].$$
\end{lemma}
\begin{proof}
The following proof will show that both sides of the given equation count all pairs $(X,f)$ such that $X$ is a subset 
of $\{1,\ldots,d-j\}$ and $f:\{1,\ldots,i\}\backslash X \rightarrow \{1,2\}$ is a 2-coloring of $\{1,\ldots,i\}\backslash X$.

On the left hand side, fix the size $\ell$ of $X \cap \{1,\ldots,i\}$.  The binomial coefficients count the number of ways
to select $X \cap \{1,\ldots,i\}$ and $X \cap \{i+1,\ldots,d-j\}$, respectively.  Finally, $2^{i-\ell}$ is the number of ways
to select $f$.

On the right hand side, set $m$ as the size of the set $Y:=f^{-1}(1) \cap \{1,\ldots,i-1\}$.  In other words, $Y$ is the 
set of elements of color 1 that are different from $i$.  There are ${i-1 \choose m}$ ways to select $Y$.  The elements of
$\{1,\ldots,i-1\}\backslash Y$ either belong to $X$ or have color 2.  If $i$ does not belong to $X$, then there are two 
ways to select the color of $i$, and $X$ is a subset of $(\{1,\ldots,i-1\}\backslash Y)\uplus \{i+1,\ldots,d-j\}$, which
may be selected ${d-m-j-1 \choose d-k}$ ways.  The elements of the remaining set 
$\{1,\ldots,d-j\}\backslash (X \uplus Y \uplus \{i\})$ must have color 2.  A similar reasoning for the case when $i$ belongs
to $X$ shows that the elements of $X \backslash \{i\}$ may be selected in ${d-m-j-1 \choose d-k-1}$ ways, completing
the proof that the right hand side counts the same set of objects as the left hand side.
\end{proof}

\begin{theorem}
\label{thrm-cwt}
Let $\mathcal{S}$ be a nonempty family of $i>1$ pairwise disjoint intervals on
$[1,d]$ and let $j$ be the number of elements in $[1,d]$ that are not in
any interval of $\mathcal{S}$. Then
$$C_{d,i,j}(x)=\sum_{\pi \in NC(d)}wt_{\mathcal{S}}(\pi).$$
\end{theorem}
\begin{proof}
Let $\mathcal{S}=\{S_1,\ldots,S_i\}$ where $S_n=[k_n,l_n]$ for $1 \le n \le i$.  Recall, if $[k_i,l_i]$ is  
wrapped, then $d$ may be an antisingleton element.
If $i=1$, then the result follows from Lemma~\ref{lem-halfQ}.  For the rest of the proof, suppose $i>1$.

Consider the weight function $wt_{\mathcal{S}}'$ defined as in the proof of Lemma~\ref{lem-halfQ}, where
it is optional to mark the special objects for some $\pi \in NC(d)$. 
It will be shown that $C_{d,i,j}(x+1)=\sum_{\pi \in NC(d)} wt_{\mathcal{S}}'(\pi)$.

To compute $\sum_{\pi \in NC(d)} wt_{\mathcal{S}}'(\pi)$, choose $\pi$ in $NC(d)$, and mark as many
nonsingleton blocks, antisingleton elements in $\mathcal{S}$, and last elements where all other elements
of the block are marked and contained in a single interval of $\mathcal{S}$ as desired.
Recall that for each interval $S_n$, $l_n$ cannot be a marked antisingleton element.
Define {\em type a} elements and {\em type b} elements as before.

Remove all type $a$ and type $b$ elements, and let $k$ be the number of elements remaining in the partition.
As before, call the noncrossing partition formed by these $k$ elements $\pi_k$.  
Recall, the weight of all possible $\pi_k$ in $NC(k)$ is $g(L_k,x+1)$, and $j \le k \le d$.  

Next, count the number of ways the $d-k$ type $a$ and type $b$ elements may be reinserted into $\pi_k$ 
at positions in $\mathcal{S}$.  Type $a$ elements cannot be inserted 
at the end of any interval in $\mathcal{S}$.  Thus, type $a$ elements can be inserted at $d-j-i$ possible locations.  
Type $b$ elements may be inserted at any position in $\mathcal{S}$,
including the last position of any interval in $S$.  Let $\ell$ be the number of type $b$ 
elements inserted at the end of some interval in $\mathcal{S}$.  Then there are ${i \choose \ell}$ ways to select
these intervals.  The remaining $d-k-\ell$ type $a$ and $b$ elements are inserted at the positions in 
$\mathcal{S}$ that are not at the end of any interval.  Hence, $\max(0,i+j-k) \le \ell \le \min(i,d-k)$.

The total possible ways to reinsert these elements is 
$$\sum_{\ell=\max(0,i+j-k)}^{\min(i,d-k)}{i \choose \ell}{d-j-i \choose d-k-\ell} 2^{d-k-\ell}=\sum_{\ell=0}^{i} {i \choose \ell}{d-j-i \choose d-k-\ell} 2^{d-k-\ell}.$$

Summing over all partitions $\pi$ of $NC(d)$ and calculating the weight of the partitions yields
\begin{equation*}
\begin{split}
\sum_{\pi \in NC(d)} wt_{\mathcal{S}}'(\pi)&=\sum_{k=j}^{d} \sum_{\ell=0}^{i} {i \choose \ell}{d-j-i \choose d-k-\ell} 2^{d-k-\ell} x^{d-k}g(L_k,x+1)  \\
&=\sum_{k=j}^{d} \sum_{\ell=0}^{i} {i \choose \ell}{d-j-i \choose d-k-\ell} 2^{i-\ell} 2^{d-k-i} x^{d-k}g(L_k,x+1).  \\
\end{split}
\end{equation*}
Apply Lemma~\ref{lem-little} to get
$$\sum_{\pi \in NC(d)} wt_{\mathcal{S}}'(\pi)=\sum_{k=j}^{d} \sum_{m=0}^{i-1} {i-1 \choose m}\left[2{d-m-j-1 \choose d-k}+{d-m-j-1 \choose d-k-1}\right]2^{d-k-i} x^{d-k}g(L_k,x+1).$$
Note, $2{d-m-j-1 \choose d-k}{d-m-j-1 \choose d-k-1}=0$ when $m>k-j$.  
Exchanging the order of summation and applying Equation~(\ref{eq:pascal})
yields
$$\sum_{\pi \in NC(d)} wt_{\mathcal{S}}'(\pi)=\sum_{m=0}^{i-1} \left(\frac{1}{2}\right)^i {i-1 \choose m}Q_{d,m+j+1}(x+1).$$
The right hand side equals $C_{d,i,j}(x+1)$ by Equation(\ref{equ-c}).
\end{proof}

\subsection{$Q_{d,k}(x)$ as the total weight of objects}
\label{sec:weight}

\begin{theorem}
\label{Thrm-double}
For $0 \le k \le d+1$, 
$$Q_{d,k}(x)=
\left\{
\begin{array}{ll}
2 \sum_{\pi \in NC(d)} wt_k(\pi) & \mbox{if $1 \le k \le d$,} \\
\sum_{\pi \in NC(d)} wt_k(\pi) & \mbox{if $k=0$ or $d+1$.} \\   
\end{array}
\right.
$$
Here $wt_k$ is the weight function given in Definition~\ref{def-wt}. 
\end{theorem}
\begin{proof}  
{\noindent\it Case 1:}
Suppose $1 \le k \le d$. This case is a direct consequence of 
Lemma~\ref{lem-halfQ} and Equation (\ref{eq:CQ}). 

{\noindent\it Case 2:}
Let $k=d+1$.  By Lemma~\ref{Thrm-wt}, 
$Q_{d,d+1}(x)=g(L_d,x)=\sum_{\pi \in NC(d)}x^{\block(\pi)}$, which gives the desired result.

{\noindent\it Case 3:}
Suppose $k=0$.
Then $\mathcal{S}=\{[1,d]^*\}$.  
Let $wt_0'(\pi)$ be the weight of $\pi$ in $NC(d)$ where marked special objects have a 
weight of $x$ and all others a weight of $1$ as defined in the proof of Lemma~\ref{lem-halfQ}.
This case be will proved by showing that $Q_{d,0}(x+1)=\sum_{\pi \in NC(d)} wt_0'(\pi).$

To compute $\sum_{\pi \in NC(d)}wt_0'(\pi)$, choose $\pi$ in $NC(d)$, and mark any nonsingleton block, 
antisingleton of $\pi$, and last element of $\pi$ where all other elements of the block have been marked.
Since $\mathcal{S}=\{[1,d]^*\}$, $d$ is an antisingleton of $\pi$ if and only if $d$ is a singleton of $\alpha(\pi)$.
Define {\em type a} and {\em type b} elements as in the proof of the Lemma~\ref{lem-halfQ}.

Let $d-\ell$ be the number of type $a$ and type $b$ elements in $\pi$.  
Then $0 \le \ell \le d$ and $\pi_{\ell} \in NC(\ell)$ as defined in Lemma~\ref{lem-halfQ}.  Suppose $\ell \neq 0$, $1$. 
Then $0 \le d-\ell < d-1$.  Summing over all partitions and ways to mark the elements yields
$\sum_{\ell=2}^{d} 2^{d-\ell}{d \choose d-\ell}x^{d-\ell}g(L_\ell,x+1).$

When $\ell=0$, all elements are type $a$ or type $b$ elements.  The total weight for 
each of these partitions is $x^d$.  Let $n$ be the number of blocks of the partition, 
and suppose $n>1$.  The type $b$ elements will determine the position of the blocks of the 
partition since each block ends with a type $b$ element.  Consider the partition in its circular representation,
and pick the location of the type $b$ elements.  
There are $\sum_{n=2}^{d} {d \choose n}=2^d-d-1$ ways to choose these positions.  
Thus, the total weight of all such partitions is $(2^d-d-1)x^d$.  
On the other hand, suppose the partition has only one block.  
Then, obviously, $\pi=(1\ldots d)$ whose weight is $x^d(x+1)$, where 
$(x+1)$ is the weight of the block, which can either be marked or unmarked.

When $\ell=1$, all elements except for one are type $a$ or type $b$ elements.
Suppose the partition has at least two blocks.  Once the unmarked position is 
chosen, there are $2^{d-1}$ ways to arrange the type $a$ and type $b$ elements.
However, if all the $d-1$ marked elements are type $a$, the partition has exactly
one block.  Hence, there are really $2^{d-1}-1$ ways to arrange the type $a$ 
and type $b$ elements, giving a total weight of $d(2^{d-1}-1)x^{d-1}$.
If the partition has one block, then $\pi=(1\ldots d)$, and its weight is $dx^{d-1}(x+1)$.

Summing over all partitions of $NC(d)$ yields a total weight of
\begin{equation*}
\begin{split}
\sum_{\pi \in NC(d)} wt_0'(\pi) &=\left[ x^d(x+1) + (2^d-d-1)x^d \right] + \left[dx^{d-1}(x+1) + d(2^{d-1}-1)x^{d-1} \right]  \\
& \quad \quad \quad \quad +  \sum_{\ell=2}^{d} 2^{d-\ell}{d \choose d-\ell}x^{d-\ell}g(L_\ell,x+1) \\
&=x^{d+1} +\sum_{\ell=0}^{d} 2^{d-\ell}{d \choose \ell}x^{d-\ell}g(L_\ell,x+1) \\
&=Q_{d,0}(x+1).
\end{split}
\end{equation*}
\end{proof}

The next two corollaries follow directly from the result of Theorem~\ref{Thrm-double}; these two properties
are illustrated in Table~\ref{table:Qpoly} for small $d$.
\begin{corollary}
\label{Lem-wt}
The coefficients of $Q_{d,k}(x)$ are nonnegative integers for $0 \le k \le d+1$. 
\end{corollary}

\begin{corollary}
The coefficients of $Q_{d,k}(x)$ are even for $1 \le k \le d$. 
\end{corollary}

\begin{table}[h]
\caption{The polynomials $Q_{d,k}(x)$ for small $d$} 
\centering 
\begin{tabular}{c r r r r r r r} 
\hline\hline 
$d$\textbackslash$k$ & 0 & 1 & 2 & 3 & 4 & 5 \\ [0.5ex] 
\hline 
0 &  $x$   &  1  & & & &   \\ 
1 &  $x^2$   &  $2x$  & 1 & & &  \\
2 &  $x^2+x^3$   &  $4x^2$  & $4x$ & $1+x$ & &   \\
3 &  $4x^3+x^4$   &  $2x^2+8x^3$  & $10x^2$ & $8x+2x^2$ & $1+4x$ &   \\
4 &  $2x^3+11x^4+x^5$   &  $12x^3+16x^4$  & $4x^2+24x^3$ & $24x^2+4x^3$ & $16x+12x^2$ & $1+11x+2x^2$  \\ [1ex] 
\hline 
\end{tabular}
\label{table:Qpoly} 
\end{table}

\section{A duality for the polynomials $C_{d,i,j}(x)$}
\label{sec:dual}

Suppose $\mathcal{S}=\{[k_1,l_1],\ldots,[k_i,l_i]\}$ as defined in Section~\ref{sec:cint} 
where $[k_n,l_n]\subseteq [1,d]$ for $1\le n \le i$.  Consider also $\mathcal{S}'$,
$[1,d]-\mathcal{S}$, and $[1,d]-\mathcal{S}'$ as defined in Section~\ref{sec:cint}.

\begin{lemma}
\label{lem-general_duality}
Let $\pi \in NC(d)$ and $\mathcal{S}$ as defined above.  
Then $wt_{\mathcal{S}}(\pi) \cdot wt_{\mathcal{S}'}(\alpha(\pi))=x^{d+1}$.
\end{lemma}
\begin{proof}
By Equation (\ref{form-wtS}),  
$wt_{\mathcal{S}}(\pi)=x^{\block(\pi)} \cdot x^{\sing_{\mathcal{S}}(\pi)} \cdot x^{\sing_{[1,d]-\mathcal{S}'}(\alpha(\pi))}$,
and $wt_{\mathcal{S}'}(\alpha(\pi))=x^{\block(\alpha(\pi))} \cdot x^{\sing_{\mathcal{S}'}(\alpha(\pi))} \cdot x^{\sing_{[1,d]-\mathcal{S}}(\pi)}$.
By definition, $\sing(\pi)=\sing_{\mathcal{S}}(\pi)+\sing_{[1,d]-\mathcal{S}}(\pi)$.
Using these definitions and applying Lemma~\ref{Lem-sum} yields
\begin{equation*}
\begin{split}
wt_{\mathcal{S}}(\pi) \cdot wt_{\mathcal{S}'}(\alpha(\pi))&=x^{\block(\pi)} \cdot x^{\sing_{\mathcal{S}}(\pi)} \cdot x^{\sing_{[1,d]-\mathcal{S}'}(\alpha(\pi))} \\
& \quad \quad \quad \cdot x^{\block(\alpha(\pi))} \cdot x^{\sing_{\mathcal{S}'}(\alpha(\pi))} \cdot x^{\sing_{[1,d]-\mathcal{S}}(\pi)} \\
&=x^{\block(\pi)+\sing(\pi)+\sing(\alpha(\pi))+\block(\alpha(\pi))} \\
&=x^{|\pi|+|\alpha(\pi)|}=x^{d+1}. 
\end{split}
\end{equation*}
\end{proof}

An immediate consequence of Lemma~\ref{lem-general_duality} is that if $wt_{\mathcal{S}}(\pi)=x^k$ for some $k$, then 
$wt_{\mathcal{S}'}(\alpha(\pi))=x^{d+1-k}$.  Note, Lemma~\ref{lem-general_duality} also holds for 
$\mathcal{S}=\emptyset$ (with $\mathcal{S}'=\{[1,d]^*\}$) and for $\mathcal{S}=\{[1,d]^*\}$
(where $\mathcal{S}'=\emptyset$) since $\sing_{\emptyset}(\pi)=0$ and $\sing_{\{[1,d]^*\}}(\pi)=\sing(\pi)$.

\begin{theorem}
\label{thrm-gen_duality}
Let $i>0$ and $0\le j \le d-i$.  Then $[x^k]C_{d,i,j}(x)=[x^{d+1-k}]C_{d,i,d-i-j}(x)$ for $0\le k \le d+1$.
\end{theorem}
\begin{proof}
By Theorem~\ref{thrm-cwt}, $C_{d,i,j}(x)=\sum_{\pi \in NC(d)}wt_{\mathcal{S}}(\pi)$ for some 
set $\mathcal{S}$ of pairwise disjoint intervals of $[1,d]$ defined as in Section~\ref{sec:cint}.  
Recall $\mathcal{S}$ consists of $i$ intervals and contains a total of $d-j$ elements in its intervals.
Define $\mathcal{S}'$ as before.  Then $\mathcal{S}'$ has $i$ intervals and contains a total of $i+j$ elements in 
its intervals.  Pick any $\pi \in NC(d)$, and let $k\ge 0$ be such that $wt_{\mathcal{S}}(\pi)=x^k$.  Define $\alpha(\pi)$ 
as before.  By Lemma~\ref{lem-general_duality}, $wt_{\mathcal{S}'}(\alpha(\pi))=x^{d+1-k}$. 
Thus, $[x^k]\sum_{\pi \in NC(d)}wt_{\mathcal{S}}(\pi) = [x^{d+1-k}]\sum_{\pi \in NC(d)}wt_{\mathcal{S}'}(\alpha(\pi))$,
which implies $[x^k]C_{d,i,j}(x)=[x^{d+1-k}]C_{d,i,d-i-j}(x)$ 
since $C_{d,i,d-i-j}(x)=\sum_{\pi \in NC(d)}wt_{\mathcal{S}'}(\alpha(\pi))$.
\end{proof}

Using these results, a similar duality result for the polynomials
$Q_{d,k}(x)$ may easily be shown.

\begin{lemma}
\label{Lem-end}
For $0 \le \ell \le \lfloor d/2 \rfloor$, $[x^\ell]Q_{d,d+1}(x)=[x^{d+1-\ell}]Q_{d,0}(x)$. 
\end{lemma}

In the proof of Proposition 2.6 in~\cite{Stanley-hvector}, Stanley showed that 
$$x^{d+1}g(L_d,1/x)=(x-1)^{d+1} + \sum_{\ell=0}^{d} 2^{d-\ell}{d \choose \ell}(x-1)^{d-\ell}g(L_\ell,x).$$
Reinterpreting this in terms of the polynomials $Q_{d,k}(x)$ yields 
$x^{d+1}Q_{d,d+1}(1/x)=Q_{d,0}(x)$, which gives the desired result.  
Additionally, if $ \ell>\lfloor d/2 \rfloor$, then $[x^\ell]Q_{d,d+1}(x)=0=[x^{d+1-\ell}]Q_{d,0}(x)$.

\begin{lemma}
\label{Lem-duality}
Let $\pi \in NC(d)$ and $0 \le k \le d+1$, then
$wt_k(\pi) \cdot wt_{d+1-k}(\alpha(\pi))=x^{d+1}$.  
\end{lemma}
\begin{proof}
If $k=0$ or $k=d+1$, the result follows directly from the proof of 
Lemma~\ref{Lem-end} and the definition of $wt_k(\pi)$.
Suppose $1\le k \le d$ and $\pi \in NC(d)$.  Let $\mathcal{S}=\{[k,d]\}$.  
Then $\mathcal{S}'=\{[d+1-k,d]\}$.  By Definition~\ref{def-wt} and Lemma~\ref{lem-general_duality}, 
$wt_k(\pi)\cdot wt_{d+1-k}(\alpha(\pi))=wt_{\mathcal{S}}(\pi)\cdot wt_{\mathcal{S}'}(\alpha(\pi))=x^{d+1}$.
\end{proof}

A direct consequence of Lemma~\ref{Lem-duality} is that if $\pi \in NC(d)$ where $wt_k(\pi)=x^\ell$ for
some $\ell$ then $wt_{d+1-k}(\alpha(\pi))=x^{d+1-\ell}$.

\begin{theorem}
\label{Thrm-duality}
Let $0 \le k \le d+1$.  Then $[x^\ell]Q_{d,k}(x)=[x^{d+1-\ell}]Q_{d,d+1-k}(x)$  for $0 \le \ell \le d+1$.
\end{theorem}
\begin{proof}
When $k$ equals $0$ or $d+1$, apply Lemma~\ref{Lem-end}.  
If $1 \le k \le d$, let $\mathcal{S}=\{[k,d]\}$; then $\mathcal{S}'=\{[d+1-k,d]\}$. 
Apply Theorem~\ref{thrm-gen_duality} to get $[x^\ell]C_{d,1,k-1}(x)=[x^{d+1-\ell}]C_{d,1,d-k}(x)$.
Thus, $[x^\ell]Q_{d,k}(x)=[x^{d+1-\ell}]Q_{d,d+1-k}(x)$ since, by
Equation (\ref{eq:CQ}), 
$Q_{d,d+1-k}(x)=2\cdot C_{d,1,d-k}(x)$.
\end{proof}

\begin{remark}
For any $d$-dimensional cubical sphere $\mathcal{P}$, the face poset is
Eulerian.  Thus, the toric $h$ polynomial
satisfies the generalized Dehn-Sommerville equations
$h(\mathcal{P},x)=x^{d+1}h(\mathcal{P},1/x)$~\cite{Stanley-hvector}.   
In terms of the polynomials $Q_{d,k}(x)$ this statement is the same as
$$\sum_{k=0}^{d+1} h_k Q_{d,k}(x)=x^{d+1}\sum_{k=0}^{d+1} h_k
Q_{d,k}\left(\frac{1}{x}\right).$$
Adin~\cite{Adin} showed that the Dehn-Sommerville
equations  holding for $\mathcal{P}$ may be restated as $h_k=h_{d+1-k}$
holding for the Adin $h$-vector.  Combining these
produces  
$$
\sum_{k=0}^{d+1} h_k Q_{d,k}(x)=\sum_{k=0}^{d+1} h_k x^{d+1}
Q_{d,d+1-k}\left(\frac{1}{x}\right).
$$ 
Since the cubical Dehn-Sommerville equations are a complete set of
linear relations even for cubical polytopes~\cite{Grunbaum},  
we know that $h_0,\ldots,h_{\lfloor \frac{d+1}{2} \rfloor}$
are linearly independent. Comparing the contributions of $h_k$ and
$h_{d+1-k}$ on both sides of the last equation yields
\begin{equation}
\label{eq:wdual}
Q_{d,k}(x)+Q_{d,d+1-k}(x)
=x^{d+1}\left(Q_{d,d+1-k}
  \left(\frac{1}{x}\right)+Q_{d,k}\left(\frac{1}{x}\right)\right). 
\end{equation} 
Conversely, it is not difficult to show that the Dehn-Sommerville
equations, stated for the Adin $h$-vector, and Equation (\ref{eq:wdual})
imply $f(\mathcal{P},x)=x^{d+1}f(\mathcal{P},1\slash x)$. Equation
(\ref{eq:wdual}) is a direct consequence of Theorem~\ref{Thrm-duality},
but Theorem~\ref{Thrm-duality} can not be derived from it. 
\end{remark}

\begin{remark}
Theorem~\ref{Thrm-duality} could also be shown directly by specializing
the proof of Theorem~\ref{thrm-gen_duality}. After that
Theorem~\ref{thrm-gen_duality} can be proved by combining
Theorem~\ref{Thrm-duality} and Equation (\ref{equ-c}) as follows: 
\begin{equation*}
\begin{split}
x^{d+1}C_{d,i,d-i-j}(1\slash x)&=\sum_{k=d+1-i-j}^{d-j}\left(\frac{1}{2}\right)^i {i-1 \choose k-1-d+i+j}x^{d+1}Q_{d,k}(1\slash x) \\
&=\sum_{k=d+1-i-j}^{d-j}\left(\frac{1}{2}\right)^i {i-1 \choose k-1-d+i+j} Q_{d,d+1-k}(x). \\
\end{split}
\end{equation*}
Set $\ell=d+1-k$ to get
$$x^{d+1}C_{d,i,d-i-j}(1\slash x)=\sum_{\ell=j+1}^{i+j}\left(\frac{1}{2}\right)^i {i-1 \choose i+j-\ell} Q_{d,\ell}(x)=C_{d,i,j}(x).$$
\end{remark}

\section*{Acknowledgments}
This work was partially supported by a grant from the Simons Foundation
(\#245153 to G\'abor Hetyei).


\end{document}